\documentclass[12pt,a4paper]{article}
\usepackage{amsmath}
\usepackage{amssymb}
\usepackage{graphicx}
\usepackage{mathtools}
\usepackage{wasysym}
\usepackage{amsfonts}
\usepackage{amsthm}
\usepackage{hyperref}
\usepackage{graphicx}
\usepackage{comment}
\usepackage[T1]{fontenc}
\usepackage[english]{babel}
\usepackage{lmodern}
\usepackage[utf8]{inputenc}
\usepackage{babel,csquotes}
\usepackage[maxbibnames=9,maxcitenames=5,backend=biber, style=alphabetic, sorting=none]{biblatex}

\AtBeginBibliography{\small}

\newtheorem{theorem}{Theorem}[section]
\newtheorem{lemma}[theorem]{Lemma}
\newtheorem{proposition}[theorem]{Proposition}
\newtheorem{definition}[theorem]{Definition}

\newtheorem{cor}[theorem]{Corollary}
\newtheorem{ex}[theorem]{Example}

\DeclareMathOperator{\Z}{\mathbb{Z}_2}
\DeclareMathOperator{\del}{\partial}
\AtBeginEnvironment{align}{\setcounter{equation}{0}}
\let\eps=\varepsilon

\begin{document}

\title{Barcode of a pair of compact exact Lagrangians in a punctured exact two-dimensional symplectic manifold}
\author{Tangi Pasquer}
\date{}

\maketitle

\begin{abstract}
In this article, we modify the classical Floer complex $CF(L_0,L_1)$ of a pair of two compact exact Lagrangian submanifolds $L_0,L_1$ of an exact symplectic 2-manifold $M$ into a $\Z[T]$-complex $CF_h(L_0,L_1)$, whose differential keeps track of how many times a pseudo-holomorphic strip passes through a distinguished point $h\in M$. We show that this complex is invariant under Hamiltonian isotopy, and we prove that its barcode, if it exists, is the same as both barcodes $\mathcal{B}(CF(L_0,L_1;M))$ and $\mathcal{B}(CF(L_0,L_1;M\setminus \{h \}))$. This allows us to extend a conjecture of Viterbo, which states that for every Hamiltonian isotopy $\phi^1_H(L)$ in $D^* L$, the spectral norm $\gamma(L,\phi^1_H(L))$ remains bounded independently of $H$, to the case of $D^* L$ with a point removed.
\end{abstract}

\tableofcontents

\section{Introduction}

\subsection{The exact Floer complex and its variation}

Andreas Floer introduced his celebrated chain complex in \cite{floer}, in order to prove a conjecture of Arnol'd concerning the minimal number of transverse intersections of the zero section of a cotangent bundle and its image under a Hamiltonian diffeomorphism. Floer theory has now become a fundamental tool in symplectic geometry, and has given rise to many problems and conjectures.

For an exact symplectic manifold $(M,\mathrm{d}\lambda)$ and two transverse exact Lagrangian submanifolds $L_0, L_1 \subseteq M$, the mod 2 Floer complex is defined by
$$CF(L_0,L_1)=\Z \langle L_0 \pitchfork L_1 \rangle,$$
with differential
$$\del x = \sum_{y \in L_0 \pitchfork L_1} n(x,y)y,$$
where $n(x,y)$ is the cardinal of $\mathcal{M}(x,y)$ modulo 2, $\mathcal{M}(x,y)$ being the set of rigid holomorphic strips in $M$ joining $x$ to $y$ with boundary on $L_0$ and $L_1$. More details on these strips can be found in Subsection \ref{2.2}.

Floer proved that $\del^2=0$, so there is a well-defined \textbf{Floer homology}
$$HF_*(L_0,L_1)=\mathrm{Ker} \del / \mathrm{Im} \del.$$

Floer also proved that for a Hamiltonian diffeomorphism $\phi$ of $M$, there is an isomorphism 
$$HF_*(L_0,L_1) \simeq HF_*(L_0,\phi(L_1)),$$
and that if $M$ is an exact symplectic manifold, with $L_0$ an exact Lagrangian and $\phi \in \mathrm{Ham}(M)$ such that $L_0 \pitchfork \phi(L_0)$, we have an isomorphism
$$HF_*(L_0,\phi(L_0)) \simeq H_*(L_0;\Z)$$
where the right hand side is the singular homology of $L_0$ with $\Z$ coefficients. \\

We now introduce a modified version of this complex in the case when $M$ is a two-dimensional symplectic manifold. This complex is related to the bulk-deformed Floer complex as defined by Fukaya, Oh, Ohta and Ono in \cite{bulk}. Let
\begin{equation*}
    CF_h(L_0,L_1)=\bigoplus_{x\in L_0 \pitchfork L_1} \Z[T]\cdot x
\end{equation*}
be the complex whose differential is $\Z[T]$-linear and is defined on the generators by
\begin{equation*}
    \del_h x = \sum_{y\in L_0 \pitchfork L_1} N(x,y)y.
\end{equation*}

The coefficients of the differential are now polynomials: for $x,y\in L_0 \pitchfork L_1$, this polynomial $N(x,y)$ is defined by
\begin{equation*}
    N(x,y)=\sum_{S\in \mathcal{M}(x,y)} T^{(S,h)} \mod 2,
\end{equation*}
where the cardinal $(S,h)\coloneqq | S^{-1}(\{h\}) |$ coincides with the algebraic intersection number between the strip $S$ and the hole $h$ (see Subsection \ref{2.2}). To put in more concrete terms, we introduce a coefficient $T^N$ whenever the strip passes through the hole $N$ times.
The fact that $\del_h^2=0$ is then a consequence of the proof of $\del^2=0$ for Floer's complex, together with a topological consideration of intersection numbers of strips with the hole.

\subsection{The theory of barcodes}

In the meantime, chain complexes and homology became widely used in several domains, and a new tool called \textbf{barcode} was introduced by Carlsson, Zomorodian, Collins, and Guibas in \cite{barcodes} to give more information than the classical homology, in particular in the case of diffeomorphic manifolds with different shapes or domains with angles and boundary.

This barcode is an invariant of filtered chain complexes: if $(C,\del)$ is a chain complex that admits a filtration by subcomplexes $C_s$, that is $C_{s'} \subseteq C_s$ if $s<s'$ and
$$C = \bigcup_{s\in\mathbb{R}} C_s,$$
then one can use the persistent homology 
$$H(C_s,\del)= \mathrm{Ker}(\del \colon C_s \rightarrow C_s) / \mathrm{Im}(\del \colon C_s \rightarrow C_s)$$
as a more precise invariant than the classical homology $H(C,\del)$.

A convenient way of representing the persistent homology is the barcode. This invariant will be properly defined in Subsection \ref{3.3}. Briefly, the number of semi-infinite bars at stage $s$ corresponds to the dimension of the homology $H(C_s,\del)$, and those bars can appear or disappear according to the choice of filtration level $s$.

\subsection{Barcodes in Floer theory}

These two mathematical tools, namely Floer homology and persistent homology, were recently combined by Polterovitch and Shelukhin, who have introduced barcodes in Floer theory in \cite{sympbarcodes1}; also see work by Usher and Zhang with \cite{sympbarcodes2}.

Here, the filtration of $CF(L_0,L_1)$ is induced by an \textbf{action function} $\ell$, which is defined on a generator $x\in L_0 \pitchfork L_1$ by 
$$\ell (x)=f_1(x) - f_0(x),$$
where $f_\beta$ is a primitive of $\lambda$ on $L_\beta$.
The fact that the strips have positive energy proves that $\del$ is action-increasing, so $\ell$ endows $CF(L_0,L_1)$ with a filtration by the subcomplexes $C_s = \{x\in CF(L_0,L_1) \colon \ell(x)>s \}$. It is then possible to talk about the barcode of the Floer complex, and some spectral invariants can be computed thanks to this barcode, such as the boundary depth or the spectral norm that go back to the work of Viterbo in \cite{vit_inv}. The reader can find a quick definition of those spectral invariants by Dimitroglou Rizell in the introduction of \cite{georgios_legend} or a more thorough development by Shelukhin in §§ 2.3, 2.4 of \cite{shelukhin18}. In short, the spectral norm of the pair $(L_0,L_1)$ can be defined here as the largest difference between the levels of two semi-infinite bars in the full barcode of $(CF(L_0,L_1),\del,\ell)$.

\subsection{A conjecture by Viterbo, and a generalization}

Viterbo has conjectured in \cite{viterbo} (Conjecture 1) that for any torus $(T^n,g)$ equipped with a Riemannian metric, and for any Hamiltonian perturbation $\phi_H^1(L)$ of $L$ inside its 1-codisk bundle $D_g^* L$, the spectral norm $\gamma(L,\phi_H^1(L))$ is bounded. Shelukhin had confirmed this conjecture in his articles \cite{shelukhin18} and \cite{shelukhin20} for certain classes of manifolds, and in particular he solved the case $n=1$. Guillermou and Vichery  have more recently proved this conjecture for homogenous spaces in \cite{guillermou22}, independently and at the same time as Viterbo in \cite{viterbo22}. 

Here, we investigate the possibility of extending the conjecture not only to codisk bundles, but to more exotic exact symplectic manifolds with boundary. Dimitroglou Rizell has proved in \cite{georgios_legend} that the conjecture neither holds for immersed Legendrians, nor for exact Lagrangians in some exact sympletic manifolds as simple as a punctured torus.

In this article, we extend the result by showing that, if an exact Lagrangian inside an exact two-dimensional symplectic manifold satisfies a bound on its spectral norm, then the same is true if one deforms the symplectic manifold by removing a point. This result relies on the following theorem, whose proof will be the main goal of this article.

\begin{theorem}\label{thm1}
Let $(M,\mathrm{d}\lambda)$ be an exact two-dimensional symplectic manifold, $L_0$ and $L_1$ two transverse compact exact Lagrangian sumbanifolds of $M$, $h\in M \setminus (L_0 \cup L_1)$ and $\phi$ a Hamiltonian diffeomorphism of $M \setminus \{h \}$ such that $\phi(L_1) \pitchfork L_0 \subseteq M \setminus \{h \}$. 

Suppose there is an action-preserving chain-isomorphism 
$$CF_h(L_0,L_1;M) \overset{\sim}{\longrightarrow} CF(L_0,L_1;M \setminus \{h \}) \underset{\Z}{\otimes} \Z[T].$$

Then, there is an action-preserving chain-isomorphism 
$$CF_h(L_0,\phi(L_1);M) \overset{\sim}{\longrightarrow} CF(L_0,\phi(L_1);M \setminus \{h \}) \underset{\Z}{\otimes} \Z[T].$$

In addition, it follows that $CF_h(L_0,\phi(L_1);M)$ has a well defined barcode, which moreover coincides with the barcodes of both regular Floer complexes $CF(L_0,\phi(L_1);M\setminus \{h \})$ and $CF(L_0,\phi(L_1);M)$.
\end{theorem}

The assumptions of the above theorem are automatically satisfied when $L_1$ is a small Hamiltonian perturbation of $L_0$. Namely, in that case, all Floer strips can be assumed to be disjoint from the point $h$. The claimed generalization of Viterbo's conjecture is the following corollary:

\begin{cor}\label{cor1}
Let $(M,\mathrm{d}\lambda)$ be an exact two-dimensional symplectic manifold, and let $L_0 \subseteq M$ be a compact exact Lagrangian submanifold of $M$. Let $L_1$ be a small Hamiltonian perturbation of $L_0$ such that $L_0 \pitchfork L_1$. Suppose that for every Hamiltonian diffeomorphism $\phi$ of $M$ such that $L_0 \pitchfork \phi(L_1)$, the spectral norm $\gamma (L_0, \phi(L_1);M)$ in $M$ is bounded by a constant independent of $\phi$. 

Let $h\in M \setminus (L_0 \cup L_1)$ be chosen sufficiently far away from the support of the Hamiltonian isotopy that takes $L_0$ to $L_1$, and let $M' \coloneqq M \setminus \{h \}$. $(M',\mathrm{d}\lambda)$ is also an exact symplectic manifold and $L_0,L_1 \subseteq M'$ are two Hamiltonian isotopic exact Lagrangians of $M'$. 

Then, for every Hamiltonian diffeomorphism $\phi$ of $M'$ such that $L_0 \pitchfork \phi(L_1)$, the spectral norm $\gamma (L_0, \phi(L_1);M')$ in $M'$
is also bounded by the same constant as above.
\end{cor}

A concrete example of a symplectic surface $M$ where the above corollary can be applied is the following:
\begin{ex}
Let $M$ be the 1-codisk bundle of the circle: $M=D^* \mathbb{S}^1=\mathbb{S}^1 \times (-1,1)$. Let $L \coloneqq \mathbb{S}^1 \times \{0 \}$ be its zero section, and let $h\in M \setminus L$.
If $\phi$ is a Hamiltonian diffeomorphism of $M \setminus \{h \}$, then the spectral norm $\gamma(L, \phi(L);M \setminus \{h \})$ in $M \setminus \{h \}$ is bounded by the same constant as for Hamiltonian diffeomorphisms in $M$ found in Shelukhin's paper \cite{shelukhin18}.
\end{ex}

\section{Definitions and conventions}

\subsection{Exact symplectic manifolds, exact Lagrangians and Hamiltonian isotopies}

Let $(M,\omega=\mathrm{d}\lambda)$ be an exact symplectic manifold of dimension 2, and $L \subseteq M$ a compact exact Lagrangian submanifold of $M$, i.e $\lambda |_{TL}$ is an exact one-form. Let $h\in M\setminus L$ be any point, which will be referred as the \textit{hole}. Let  $M'\coloneqq M\setminus \{h\}$ and $\omega' \coloneqq \omega_{|\Lambda^2 TM'} $. Then, $(M',\omega')$ is an exact symplectic manifold as well, and $L$ is still an exact Lagrangian of $(M',\omega')$.

Finally, let $L_0$ be another compact exact Lagrangian of $M'$ such that $L \pitchfork L_0$. Let $L_1$ be a Lagrangian that is Hamiltonian isotopic to $L_0$ by a Hamiltonian isotopy supported inside $M'$, and such that $L \pitchfork L_1$. We can then choose a Hamiltonian isotopy $\phi_t \coloneqq \phi^t_H \colon [0,1] \times M' \rightarrow M'$ such that $L_t\coloneqq \phi_t(L_0)$ is transverse to $L$ for all times $t\in[0,1]$ except for a finite set $\{t_1,\dots,t_n \}$, with $0<t_1<\dots<t_n<1$. That is a well-know fact that such Hamiltonian isotopies always exist for any pair of Hamiltonian isotopic Lagrangians. For example, a more general statement can be found as Lemma 3.3 in \cite{floer}. The argument is that, in a neighborhood of a point $x\in L\cap L_t$, the Lagrangian $L_t$ can be identified with the graph $\mathrm{gr}(\mathrm{d}f) \subseteq T^* L$ of an exact one-form on $L$, and $f$ can be modified to be a Morse function.

Let us remind ourselves of the ODE that defines the Hamiltonian flow:
\begin{equation*}
    \del_t \phi_t=X_t \circ \phi_t.
\end{equation*}
The vector field $X_t$ here satisfies 
\begin{equation*}
    \iota _{X_t} \omega = \omega (X_t, -) = \mathrm{d}H_t,
\end{equation*}
where $H\colon [0,1] \times M' \rightarrow \mathbb{R}$ is a smooth function called \textbf{Hamiltonian}.


The Lagrangian $L$ is exact, so we can choose a primitive $f$ of $\lambda |_{TL}$. For the family $L_t, \ t\in[0,1]$, it follows from Cartan's formula 
$$\frac{\mathrm{d}}{\mathrm{d}t} (\phi_t^*\lambda)=\mathrm{d}(\phi_t^* H_t) + \mathrm{d}(\phi_t^*\lambda(X_t))$$
that $L_t$ is a family of exact Lagrangian submanifolds with continuously varying primitives $f_t$ of $\lambda|_{TL_t}$.
We then define the \textbf{action} function $\ell_t \colon L \cap L_t \rightarrow \mathbb{R}$ by
$$\ell_t(x) \coloneqq f_t(x) - f(x).$$

For the times $t\in[0,1] \setminus \{t_1,\dots,t_n\}$ of transverse intersection, $L \pitchfork L_t \subseteq L$ is a 0-submanifold of $L$, hence a finite set since $L$ is compact.

\subsection{Classical and modified Floer complex} \label{2.2}

Let $\beta \in \{0,1\}$, and let $M_1 \coloneqq M$ and $M_0 \coloneqq M'$. A \textbf{pseudo-holomorphic strip} in $M_\beta$ joining $x$ to $y$ with boundary on $L$ and $L_t$ is a smooth function 
$$u \colon \mathbb{R} \times [0,1] \longrightarrow M_\beta$$
such that $\mathrm{d}u \circ j = J_\beta \circ \mathrm{d} u$, where $j$ is the canonical complex structure on $\mathbb{C}$ and $J_\beta$ is a fixed $\omega $-compatible almost complex structure on $M_\beta$, i.e such that $g(v,w) \coloneqq \omega (v,J_\beta w)$ defines a Riemannian metric. In addition, $u$ must satisfy the boundary conditions $u(\mathbb{R}\times \{0\}) \subseteq L$, $u(\mathbb{R}\times \{1\}) \subseteq L_t$, and $u(s,r) \underset{s \to +\infty}{\longrightarrow} x$, $u(s,r) \underset{s \to -\infty}{\longrightarrow} y$ for all $r\in[0,1]$.

We denote $\mathcal{M}_t^\beta(x,y)$ the set of all rigid holomorphic strips in $M_\beta$ of index 1 joining $x$ to $y$ with boundary on $L$ and $L_t$, modulo holomorphic reparametrization. \textit{Rigid} means that the strip is transversely cut out as a solution, and that its expected dimension is zero. We will not expand further on this, but instead use the well-known fact that when the symplectic manifold is two-dimensional, this is equivalent to the strip being an immersion up to and including the boundary, with convex corners. See \cite{combinatorial}, Lemma 12.3, for the relation between these notions. We will abbreviate the name of the elements of $\mathcal{M}_t^\beta(x,y)$ by \textit{strips}. \\

Let $\Z=\mathbb{Z}/2\mathbb{Z}$. For $\beta\in\{0,1\}$ and $t\in[0,1] \setminus \{t_1,\dots,t_n\}$, let $C_\beta(t) \coloneqq CF(L,L_t;M_\beta) \coloneqq (C(t),\del_\beta(t))$ be the Floer complex associated with the two transverse Lagrangians $L$ and $L_t$, seen as Lagrangian submanifolds of $M_\beta$:
\begin{equation*}
    C(t) \coloneqq CF(L,L_t)=\bigoplus_{x\in L \cap L_t} \Z\cdot x,
\end{equation*}
whose differential $\del_\beta(t)$ is defined by
\begin{equation*}
    \del_\beta(t) x = \sum_{y\in L \cap L_t} n_t^\beta(x,y)y,
\end{equation*}
where $n_t^\beta(x,y)=| \mathcal{M}_t^\beta(x,y) | \mod 2$. It is a consequence of Stokes' formula and of the exactness of the Lagrangians that 
\begin{equation}\label{stokes}
	\ell_t(y)>\ell_t(x)
\end{equation}
whenever there is a strip joining $x$ to $y$ in $M$.

We then extend the definition of the action $\ell_t$ for every element of $C(t)$: 
$$\ell_t \left( \sum_{i=1}^p \lambda_i x_i \right) \coloneqq \inf \{ \ell_t(x_i) \ | \ 1 \leq i \leq p, \ \lambda_i \neq 0 \},$$
where $(x_1,\dots,x_p)$ is the list of the generators of $C(t)$ and $\lambda_1,\dots,\lambda_p \in \Z$. We use here the convention $\inf \varnothing =+\infty$. \\

Since strips are orientation preserving immersions, they have a preferred behaviour near their inputs and outputs, as shown in  Fig. \ref{strip_or}: if the strip lies on a blue corner 2 or 4, it means that the strips begins at this point; otherwise, it should end at this point.

\begin{figure}[ht]
\begin{center}
\includegraphics[scale = .25] {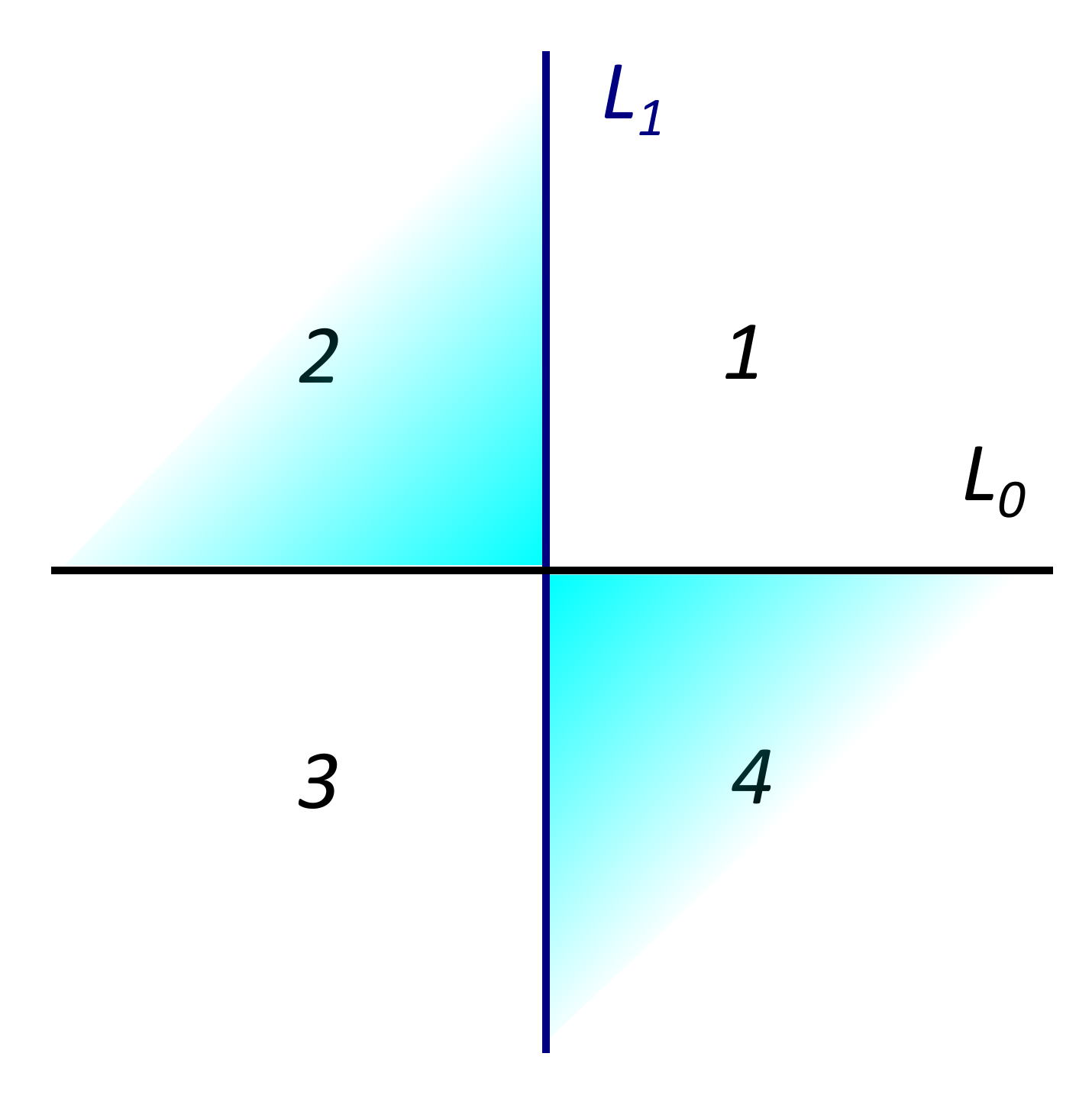}
\end{center}
\caption{Convention for the orientation of the holomorphic strips}
\label{strip_or} 
\end{figure}

The fact that $\del_\beta(t)^2=0$ lies in the heart of Floer theory, and has been proved in a more general setting in \cite{floer} in the case of aspherical symplectic manifolds, that is when $\pi_2(M,L)=0$. Note that an exact Lagrangian inside an exact symplectic surface necessarily satisfies this assumption. Floer's proof relies on Gromov's compactness theorem, which is valid for all open 2-dimensional symplectic manifolds, and hence in particular after removing a point in a symplectic surface. The reason why Gromov's compactness holds is that the maximum principle prevents holomorphic curves inside the symplectic surface from escaping to infinity. Such curves can thus be confined to some a priori given compact subset.

In addition, there is a combinatorial proof of $\del^2=0$ for many symplectic surfaces in \cite{combinatorial}. \\


The name "chain complex" may seem to be an abuse of language because we have not defined such thing as a grading for the Floer complex, but it is actually possible to define a degree $\mathrm{deg} (x)$ for the points $x\in L_0 \pitchfork L_1$ such that $\mathrm{deg} (\del x)= \mathrm{deg} (x) - 1$. The reader can refer to \cite{georgios_legend}, Subsection 2.5, or to \cite{auroux}, Subsection 1.3 for a quick presentation of the grading for the Floer complex. \\

We shall now define the modified Floer complex which is central to this article: let
\begin{equation*}
    D(t) \coloneqq CF_h(L,L_t) \coloneqq CF_h(L,L_t;M) \coloneqq \bigoplus_{x\in L \cap L_t} \Z[T]\cdot x
\end{equation*}
be the complex with differential $\del_h(t)$ defined on the generators by
\begin{equation*}
    \del_h(t) x = \sum_{y\in L \cap L_t} N_t(x,y)y.
\end{equation*}
The polynomials $N_t(x,y)$ are defined by
\begin{equation*}
    N_t(x,y)=\sum_{S\in \mathcal{M}_t^1(x,y)} T^{(S,h)} \mod 2,
\end{equation*}
where $(S,h)$ denotes the algebraic intersection number between the strip $S$ and the hole $h$. As the pseudo-holomorphic strips are orientation-preserving immersions, the algebraic intersection number is non-negative and then coincides with $| S^{-1}(\{h\}) |$: therefore, we only have to count the number of times that the strip passes over $h$.

Again, these modules are indeed chain complexes: the fact that $\del_h(t)^2=0$ is a consequence of the proof of $\del_1(t)^2=0$, and it will not be further detailed in this article. In short, since pseudo-holomorphic strips are orientation-preserving, the algebraic intersection number $| S^{-1}(\{h\}) |$ does not change in the 1-parameter families of index 2 strips. 

We extend the above definition of $\ell_t$ to all of $D(t)$ by the same formula. 

For simplicity, we will denote $\langle -,- \rangle$ the canonical inner products on $C(t)$ and $D(t)$: for $x,y \in L\cap L_t, \ \langle x,y \rangle = \delta_{x,y}$.

\section{The barcode of a piecewise continuous family of complexes}

We are going to define the main features of the theory of barcodes as presented in the second part of \cite{georgiosbarcode}.  

We will only consider finitely generated free chain complexes, but  will often omit this precision and just write complexes.

In this whole section, $k$ is a field, and $R$ denotes either $k$ or $k[T]$. All the chain complexes will be $R$-modules.

\subsection{Filtered chain complexes}

\begin{definition}
A \textbf{filtered chain complex} is a (finitely generated and free) chain complex $(C, \del)$ over the ring $R$ with a function $\ell\colon C \rightarrow \mathbb{R} \cup \{ +\infty \}$ called \textbf{action} such that
\begin{itemize}
	\item $\ell (x)=+ \infty$ if and only if $x=0$,
	\item $\forall x \in C, \ \forall \lambda \in R, \  \ell(\lambda x) = \ell(x),$
	\item $\forall x,y \in C,  \  \ell(x+y) \geq \min \{\ell(x),\ell(y)\}$,
	\item $\forall x \in C,  \  \ell(\del x) \geq \ell(x)$.
\end{itemize}
\end{definition}

As the name suggests, these complexes are filtered. Namely, there is a natural filtration by the subcomplexes $C_{>s} \coloneqq \{ x\in C \colon \ell(x)>s\}$, for $s\in\mathbb{R}$, i.e these subcomplexes verify $C_{>s'} \subseteq C_{>s}$ for $s<s'$, and $C=\bigcup_{s\in\mathbb{R}} C_{>s}$. The ring $R$ being a PID, these subcomplexes are free too.

An isomorphism in the category of filtered chain complexes $f \colon (C,\del,\ell) \overset{\sim}{\longrightarrow} (C',\del',\ell')$ is a chain-isomorphism $f \colon (C,\del) \overset{\sim}{\longrightarrow} (C',\del')$ that preserves the filtration, such that \textit{its inverse} $f^{-1}$ is also filtration-preserving (this condition is not automatic). Equivalently, it is a chain-isomorphism which preserves \textit{the action}.

Such a complex has a distinguished class of bases, called \textbf{compatible bases}: these are the bases $(a_1,\dots,a_n)$ of $C$ such that 
$$\forall (\lambda_1,\dots,\lambda_n)\in R^n, \quad \ell\left( \sum_{i=1}^n \lambda_i a_i\right)=\inf_{\lambda_i \neq 0} \ell(a_i).$$
This definition is compatible with our convention $\inf \varnothing = + \infty$.

\begin{lemma}
Any filtered complex possesses a compatible basis.
\end{lemma}

\begin{proof}
As presented in \cite{georgiosbarcode}, one can consider the quotient modules 
$$Q_{s,\eps} \coloneqq C_{>s-\eps} / C_{>s}$$
where $C_{>s} \coloneqq  \{x\in C \colon \ell(x)>s\}$ is a subcomplex of $C$ and $\eps>0$.
There is only a finite number of values of $s$ for which the quotients $Q_{s,\eps}$ are non-zero. Moreover, they are free by the axioms of the filtered chain complex. To see this, note that since $R$ is a PID it suffices to verify that the quotient is torsion free. This follows since scalar multiplication by the elements of $R$ is action-preserving. For all $\eps$ small enough, taking a basis of each $Q_{s,\eps}$ and choosing representatives of their classes yield a compatible basis for $C$.
\end{proof}

\begin{definition}\label{pwc}
A \textbf{piecewise continuous filtered chain complex} is a family of filtered chain complexes $(C(t), \del_t, \ell_t), \ t\in [0,1]$, with a finite collection of times $0<t_1<\dots<t_n<1$, such that the following conditions hold:
\begin{enumerate}
    \item for $t,t'\in (t_i,t_{i+1}), \ (C(t),\del_t)$ is chain-isomorphic to $(C(t'),\del_{t'})$ through a preferred isomorphism $\Phi_{t,t'} \colon C(t) \xrightarrow[]{\sim} C(t')$ that preserves compatible bases, with continuous action $\ell_t$ in the following sense: 
    $$\forall x \in C(t), \quad \ell_{t'}(\Phi_{t,t'}(x)) \xrightarrow[t' \to t]{} \ell_t(x).$$
    Moreover, the family is functorial, meaning that if $t,t',t^{\prime \prime} \in (t_i,t_{i+1}),$ we have
    $$\Phi_{t,t''} = \Phi_{t',t''} \circ \Phi_{t,t'}$$
    \item for each $t_i$ with $i\in\{1,\dots,n\}$, one of the following \textbf{simple bifurcations} occur:
    \begin{itemize}
        \item \textbf{Birth:} there is a 2-dimensional $R$-complex $(S,\del,\ell_t)$, such that $\ell_t$ is defined to satisfy $\ell_t(c) = \ell_t(d)$ and to be independent of time, such that $(c,d)$ is a compatible basis and $\del c=d$ (note that this differential is thus not strictly action-increasing), and a family of chain-isomorphisms that send compatible bases to compatible bases
        $$(C(t_i - \eps),\del_{t_i-\eps},\ell_{t_i-\eps}) \oplus (S,\del,\ell_{t_i-\eps}) \overset{f_\eps}{\longrightarrow} (C(t_i + \eps),\del_{t_i+\eps},\ell_{t_i+\eps})$$
        for all $\eps > 0$ small enough, such that 
        $$\forall x \neq 0, \ \ell_{t_i+\eps}(f_\eps(x))-\ell_{t_i-\eps}(x) \underset{\eps \to 0}{\longrightarrow} 0.$$ 
        Finally, the above preferred isomorphism $\Phi_{t,t'}$ extends to $(t_{i-1},t_i]$.
        \item \textbf{Death:} the family $C(-t)$ has a birth at $-t_i$.
        \item \textbf{Handle-slide:} for all $\eps>0$ small enough, there is a non-canonical chain-isomorphism that preserves compatible bases 
        $$H_i(\pm \eps) \colon (C(t_i),\del_{t_i},\ell_{t_i}) \longrightarrow (C(t_i \pm \eps), \del_{t_i \pm \eps},\ell_{t_i \pm \eps}),$$
        that is moreover represented by an upper-triangular matrix in two compatible bases ordered by decreasing action level. Finally, we have
         $$\forall x \neq 0, \ \ell_{t_i \pm \eps}(H_i(\pm \eps)(x))-\ell_{t_i}(x) \underset{\eps \to 0}{\longrightarrow} 0.$$
    \end{itemize}
\end{enumerate}
\end{definition}

This terminology comes from the study of the Morse complex. In our case, the handle-slide never comes without a birth/death, as we will see later. However, for higher dimensions, these bifurcations may happen independently.

Note that Condition (1) implies that even though the differential remains the same in $C(t)$, the actions of its elements change, and so does the filtered isomorphism class of $C(t)$.

\subsection{Barannikov decompositions}

The most well-studied tool for understanding chain complexes is their homology: since $\del^2=0$, the submodule of the \textbf{boundaries} $B=\mathrm{Im} \del$ is included in the submodule of the \textbf{cycles} $Z=\mathrm{Ker} \del$, and we can consider the quotient module $H(C,\del)\coloneqq Z/B$, which is called the \textbf{homology} of $C$. As we have the action $\ell$, we can even define the \textbf{persistent homology groups} $H(C_{>s},\del)$ for $s\in \mathbb{R}$. In the next subsection we will construct the barcode of a filtered complex, that will turn out to be a useful invariant for visualizing a filtered complex. We will define the barcode using the so-called Barannikov decomposition, which was first considered in \cite{barannikov}.

\begin{definition}
A \textbf{Barannikov basis} of a filtered complex $(C,\del)$ is a \textit{compatible} basis $(a_1,\dots,a_n,b_1,\dots,b_n,c_1,\dots,c_m)$ such that:
\begin{itemize}
    \item for $1\leq i\leq n, \ \del a_i=b_i,$
    \item $(b_1,\dots,b_n)$ is a basis of $B$,
    \item $(b_1,\dots,b_n,c_1,\dots,c_m)$ is a basis of $Z$.
\end{itemize}
\end{definition}

A Barannikov basis allows one to decompose the complex into more simple 1- and 2-dimensional subcomplexes: if $(a_1,\dots,a_n,b_1,\dots,b_n,c_1,\dots,c_m)$ is Barannikov, then 
$$(C,\del)=\bigoplus_{i=1}^n (Ra_i\oplus Rb_i,\del_i) \oplus \bigoplus_{j=1}^m (Rc_j,0)$$
where $\del_i$ is the restriction of $\del$ to $Ra_i\oplus Rb_i$.
This decomposition allows one to easily compute the homology of $(C_{>s},\del)$: it is freely generated by the cycles $c_j$ such that $\ell(c_j)>s$, but also by the boundaries $b_i$ such that $\ell(b_i)>s$ and $\ell(a_i) \leq s$, as in $C_{>s}$ the chain $a_i$ does not exist yet. 

Therefore, a necessary condition for a complex $C$ to possess a Barannikov basis is that all the homolgy groups $H(C_{>s},\del)$ are free $R$-modules. However this condition is not always the case if $R$ is not a field, as shown by the example below.

\begin{ex}
Let $R=k[T]$, and let $\del$ be the $R$-linear map on $C=Ra\oplus Rb$ that satisfies $\del a=Tb$ and $\del b=0$. Then $b$ is a cycle but not a boundary, although $Tb$ is a boundary. Thus, $H(C,\del)$ has torsion and is not a free $R$-module. 
\end{ex}

\begin{lemma}[Barannikov]\label{barannikov}
\hfill
\begin{enumerate}
    \item Any finite-dimensional filtered chain complex over $k$ has a Barannikov decomposition.
    \item For a finite-dimensional filtered chain complex over $R$ with a compatible basis $(x_1,\dots,x_n)$ such that $\ell(x_1) < \dots < \ell(x_n)$, any two Barannikov bases (if they exist) are related by an upper-triangular change of basis matrix.
\end{enumerate}
\end{lemma}

The proof of the previous statement can be found in \cite{barannikov}. In particular, its first statement guarantees the existence of a Barannikov basis on the complex $C(t)$ for every $t\in [0,1]$. Nevertheless, if $R$ is not a field, then a Barannikov basis does not necessarily exist, as shown in the above example.
Moreover, the second statement is not sufficient for us since we need to handle the cases where there is redundancy in the action levels.

Let us first reformulate the existence of the Barannikov decomposition in the case of $k[T]$:

\begin{proposition}\label{standard}
Let $R=k[T]$, and let $(C_T,\del_T,\ell_T)$ be a filtered chain complex over $R$. Then, $(C_T,\del_T,\ell_T)$ has a Barannikov basis if and only if there exists a filtered chain complex $(C,\del,\ell)$ over $k$ such that there is an action-preserving chain-isomorphism
$$(C_T,\del_T,\ell_T) \overset{\sim}{\longrightarrow} (C,\del,\ell) \underset{k}{\otimes} k[T].$$
\end{proposition}

\begin{proof}
If $B=(a_1,\dots,a_n,b_1,\dots,b_n,c_1,\dots,c_m)$ is a Barannikov basis of the filtered chain complex $(C_T,\del_T,\ell_T)$, then we define the formal $k$-module
$$C \coloneqq k \langle B \rangle,$$
which we equip with the $k$-linear differential $\del \colon C \rightarrow C$ such that $\del a_i = b_i$, $\del b_i = 0$ and $\del c_j = 0$, and the action $\ell$ such that $B$ is a compatible basis of $(C,\del,\ell)$ and $\ell(x)=\ell_T(x)$ for each element $x$ of the basis $B$.
It is now obvious that $(C_T,\del_T,\ell_T)$ and $(C,\del,\ell) \underset{k}{\otimes} k[T]$ are isomorphic as filtered chain-complexes over $k[T]$.

Conversely, if we have an action-preserving chain-isomorphism
$$(C,\del,\ell) \underset{k}{\otimes} k[T] \overset{f}{\longrightarrow} (C_T,\del_T,\ell_T),$$
we can find a Barannikov basis $B$ for $(C,\del,\ell)$ thanks to Lemma \ref{barannikov}, and taking $f(B)\subseteq C_T$ we get a Barannikov basis for $(C_T,\del_T,\ell_T)$.
\end{proof}

\subsection{Towards the barcode} \label{3.3}

The next definition identifies the class of filtered chain-complexes we will be working with.

\begin{definition}
A filtered chain complex of $R$-modules $(C,\del,\ell)$ satisfying one of the equivalent conditions in Proposition \ref{standard} is called a \textbf{standard} complex. 
\end{definition}

The following proposition will provide a uniqueness result for the Barannikov basis of a standard complex.

\begin{proposition}
If $(C_T,\del_T,\ell_T)$ is a standard filtered $k[T]$-complex, there is a filtered $k$-complex $(C,\del,\ell)$, which is unique up to action-preserving chain-isomorphism, such that
$$(C_T,\del_T,\ell_T) \simeq (C,\del,\ell) \underset{k}{\otimes} k[T].$$

Moreover, the form of the Barannikov decomposition is unique for a standard $k[T]$-complex.
\end{proposition}

\begin{proof}
$(C_T\del_T,\ell_T)$ begin standard, we can choose a filtered $k$-complex $(C,\del,\ell)$ such that we have an action-preserving chain-isomorphism
$$(C_T,\del_T,\ell_T) \overset{f}{\longrightarrow} (C,\del,\ell) \underset{k}{\otimes} k[T].$$

First, we define the $k$-vector space
$$\tilde{C} \coloneqq C_T / T \cdot C_T,$$
where $T \cdot C_T$ denotes the submodule that is induced by the ideal $\langle T \rangle$, which turns out to be a subcomplex. Let $p \colon C_T \twoheadrightarrow \tilde{C}$ be the natural projection.
$\tilde{C}$ naturally comes with a differential $\tilde{\del}$ such that 
$$\forall x \in C_T, \ \tilde{\del} (p(x)) = p(\del x).$$

Then, we define an action $\tilde{\ell}$ on $\tilde{C}$ by 
$$\tilde{\ell}(\tilde{x}) \coloneqq \sup_{p(x)=\tilde{x}} \ell_T(x).$$
The verification of the axioms is straightforward and will not be detailed here. 

One should note that the same construction, when applied to $(C,\del,\ell) \otimes k[T]$, gives back the complex $(C,\del,\ell)$. Let $p' \colon C \otimes k[T] \twoheadrightarrow C$ be the natural projection through this natural identification.


It is obvious that $f(T \cdot C_T)=T \cdot (C \otimes k[T])$, so the quotient map $\tilde{f}$ is a linear bijection
$$(\tilde{C},\tilde{\del},\tilde{\ell}) \overset{\tilde{f}}{\longrightarrow} (C,\del,\ell).$$
The definition of $\tilde{\del}$ shows that $f$ is a chain-map. To see that it is action-preserving, let $\tilde{x} \in \tilde{C}$ and $x\in C_T$ such that $p(x)=\tilde{x}$: we have
$$\ell \left( \tilde{f}(\tilde{x}) \right) = \ell \left(p'(f(x)) \right)=\sup_{y\in C \otimes k[T]} \ell \left(f(x) + Ty \right) = \sup_{z\in C_T} \ell_T \left(x+ Tz \right)=\tilde{\ell}\left(\tilde{x} \right),$$ 
since $\tilde{f} \circ p = p' \circ f$.
This shows that $(C,\del,\ell)$ is chain-isomorphic to $(\tilde{C},\tilde{\del},\tilde{\ell})$, the latter filtered complex being independent of the choice of $(C,\del,\ell)$, thus concluding the proof of the first statement. The second one follows from the first bullet point of Lemma \ref{barannikov}, since a Barannikov basis for $C$ is sent to a Barannikov basis for $\tilde{C}$ under the projection map $p$.
\end{proof}


We are now able to define the barcode of a standard complex:

\begin{definition}
\hfill
\begin{enumerate}
    \item A \textbf{barcode} is a finite (multi-)set of semi-closed intervals of the form $(e,s]$, called bars, where $s\in\mathbb{R}$ is the starting point, and $e\in\mathbb{R\cup \{-\infty\}}$ is the endpoint of the bar.
   \item The barcode of a standard complex $(C,\del,\ell)$, denoted by $\mathcal{B}(C,\del,\ell)$, is defined as the barcode with the bars $(\ell(a_i),\ell(b_i)]$ and $(-\infty,\ell(c_j)]$, where $(a_1,\dots,a_n,b_1,\dots,b_n,c_1,\dots,c_m)$ is a Barannikov basis of $(C,\del,\ell)$.
\end{enumerate}
\end{definition}

\section{Invariance of $CF_h(L,L_t)$ under Hamiltonian isotopies}

Let us remind ourselves of the family of chain complexes $(D(t),\del_h(t))$: we have
$$D(t)=C_h(L_0,L_t;M)=\bigoplus_{x\in L \cap L_t} \Z[T] \cdot x,$$
with 
$$\langle \del_h(t)(x),y\rangle = \sum_{S \in \mathcal{M}_t(x,y)}  T^{(S,h)}$$
being defined for $t \in [0,1]$ such that $L \pitchfork L_t$. Here, $\mathcal{M}_t(x,y)$ denotes the $t$-dependent set of rigid holomorphic strips in $M$ joining $x$ to $y$ and $(S,h)$ the algebraic intersection number of the strip $S$ and the point $h$. Moreover, we have defined an action $\ell_t$ for the complex $D(t)$ by
$$\ell_t \left( \sum_{i=1}^p \lambda_i x_i \right) \coloneqq \inf \{ \ell_t(x_i) \ | \ 1 \leq i \leq p, \ \lambda_i \neq 0 \},$$
where $(x_1,\dots,x_p)$ is the list of the generators of $D(t)$. One should note that the differential $\del_h(t)$, as well as $\del_\beta(t)$ for $\beta \in \{0,1\}$, is strictly action-increasing with this action $\ell_t$. Indeed, it is a consequence of the definition of $\ell_t$ and Equation \eqref{stokes}. 

We denote $\{t_1<\dots<t_n\}$ the set of all times $t$ in which $L$ is not transverse to $L_t$.
We define 
$$\Delta t \coloneqq \inf_{1\leq i<j \leq n} t_j-t_i.$$

Our first milestone for the proof of Theorem \ref{thm1} is the following statement.

\begin{theorem}\label{thm}
After a slight continuous perturbation of the actions of the canonical basis elements of the family $D(t)$ of complexes, one obtains a family $\tilde{D}(t)$ of complexes which is a piecewise continuous family of filtered chain complexes, where the canonical identification of basis elements of $D(t)$ and $\tilde{D}(t)$ moreover gives an action-preserving chain-isomorphism away from some small neighborhood of the non-transverse moments $t_1,\dots,t_n$.
\end{theorem}

The goal of this part is to prove Theorem \ref{thm}. The family $\tilde{D}(t)$ will be defined in Subsection \ref{sub}, and the proof of Theorem \ref{thm1} will be completed in Subsection \ref{sec}.

\subsection{First observations}

We start by breaking down the transformations undergone by the complex $D(t)$ through time variations into elementary steps as pictured in Fig. \ref{fig1}, which describes a birth moment in which precisely two new intersections $c$ and $d$ are born. 

The following discussion has been inspired by \cite{chekanov}; although the context of Floer homology is much simpler than the differential graded algebra used by Chekanov, there is a deep relationship between these two setups. Indeed, a pair of Lagrangian embeddings lifts to a pair of Legendrian embeddings in the contactization of the symplectic manifold, and the Floer complex of the pair of Lagrangians can then be recovered from Chekanov's DGA of the pair of Legendrians. For more details, the reader can refer to \cite{georgios_legend}.

\begin{figure}[ht]
\begin{center}
\includegraphics[scale = .25] {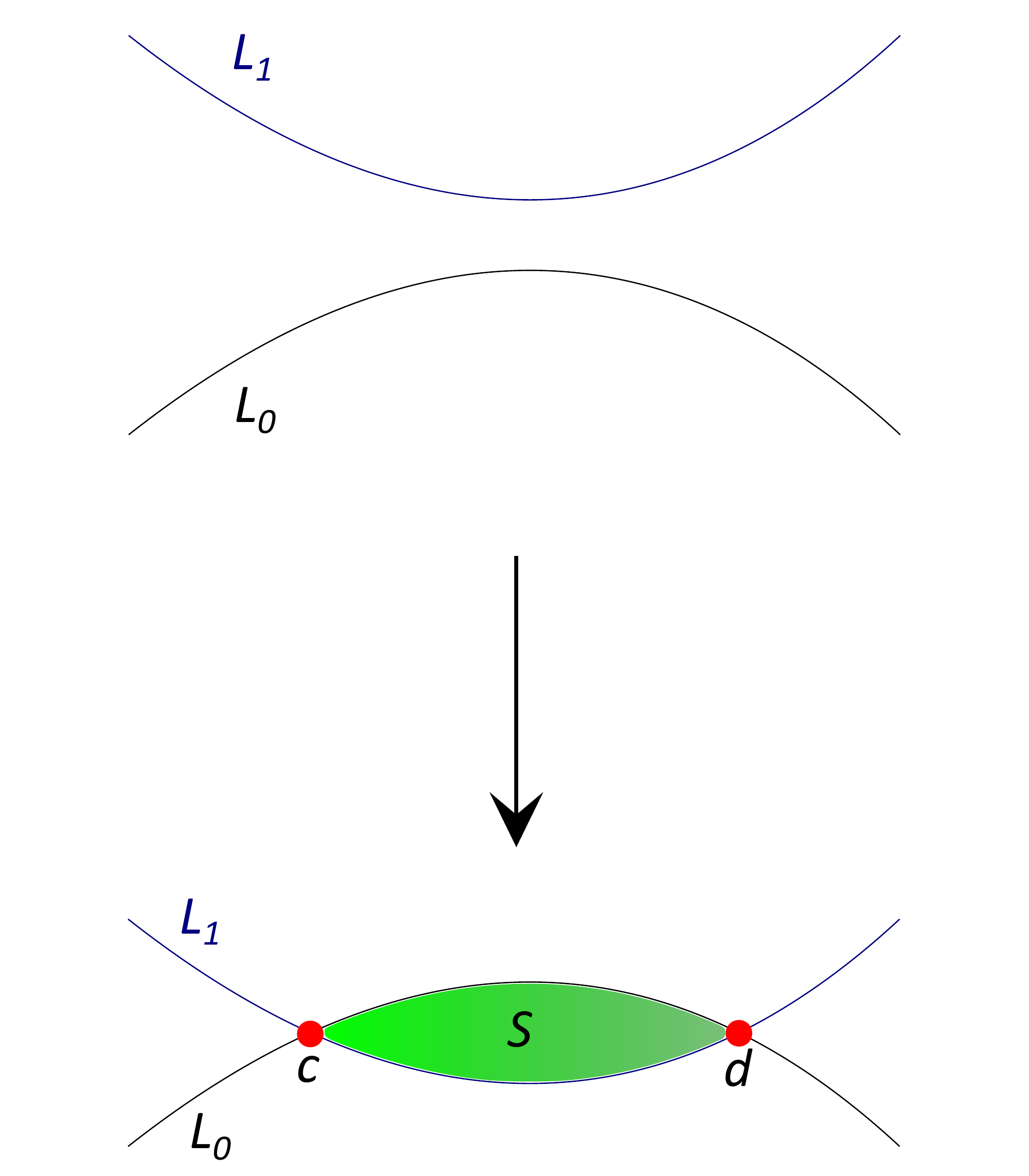}
\end{center}
\caption{An elementary crossing}
\label{fig1}
\end{figure}

Though several crossings could happen simultaneously at a time $t_0$, we can let them happen successively by adding a small perturbation $\delta(x,t)$ to the Hamiltonian, with support in $(t_0-\eps,t_0+\eps) \times M'$, with $\eps<\Delta t.$ Thus, it is correct to assume that the bifurcations are  completely described by Fig. \ref{fig1}. 

We now consider the case of a birth moment at $t=t_0$, which means that there are two new transverse intersection points $c$ and $d$ for $t>t_0$ as shown in Fig. \ref{fig1}.
Denote $\del_0 \coloneqq \del_h(t_0-\eps)$ and $\del_1 \coloneqq \del_h(t_0+\eps)$.
Let us order the generators of $D(t_0+\eps)$, that is the elements of $L\cap L_{t_0+\eps}$, by decreasing actions:
$$\ell(b_m)\geq\dots>\ell(b_1)\geq \ell(d)> \ell(c)\geq \ell(a_1)\geq\dots\geq \ell(a_n).$$
Here we assume that this order does not change between $t_0-\eps$ and $t_0+\eps$, so that we can omit the indices on $\ell$. Again, a small Hamiltonian perturbation allows us to make this assumption without loss of generality.

We first focus on the new born points (Lemma \ref{lemma3}), but also on the strips that survive through the crossing (Proposition \ref{prop3}).

\begin{lemma}\label{lemma3} There is a unique strip joining $c$ to $d$. Moreover, it does not pass through the hole, so we have
\begin{equation*}
    \langle \del_1 c, d\rangle = 1.
\end{equation*}
\end{lemma}

\begin{proof}
At $t=t_0$, $L$ and $L_t$ are tangent in one point $e$ of action $l$, such that $\ell_t(c),\ell_t(d) \xrightarrow[t\rightarrow t_0]{} l$.
We can perturb $H$ around $t_0$ so that only a neighborhood $U$ of $e$ is changed over time. The only strip that stays in $U$ is $S$, so any other strip has an area greater than some constant $C>0$ during this process. However, $\ell_t(d)-\ell_t(c) \xrightarrow[t\rightarrow t_0]{} 0$, and $\ell_t(d)-\ell_t(c)$ is the area at $t$ of any strip joining $c$ to $d$. Therefore, there cannot be a strip joining $c$ to $d$ other than $S$.
\end{proof}

\begin{proposition}\label{prop3}
The strips joining two higher-action generators $b_k$ and $b_l$, and those joining two lower-action generators $a_i$ and $a_j$ are not disturbed by the birth. More precisely, we have
\begin{equation}\label{eq5}
    \forall k\in \{1,\dots,m\}, \quad \del_0 b_k=\del_1b_k.
\end{equation}
\begin{equation}\label{eq6}
    \forall i,j \in \{1,\dots,n\}, \quad \langle \del_1 a_i,a_j \rangle = \langle \del_0 a_i,a_j \rangle.
\end{equation}
\end{proposition}

\begin{proof}
An old strip that disappears after the birth had to pass through the green part in Fig. \ref{fig0}. Thus, it gives birth to two new strips: one of them begins on $c$, and thus ends on some generator $b_k$, and the other one ends on $d$, and thus begins on some point $a_i$. 

Therefore, all the strips that disappeared were those joining some $a_i$ to some $b_k$, and those that appear must have endpoints or starting points on $c$ or $d$. This proves the proposition, whose the two identities are direct consequences.
\end{proof}

\begin{figure}[ht]
\begin{center}
\includegraphics[scale = .25] {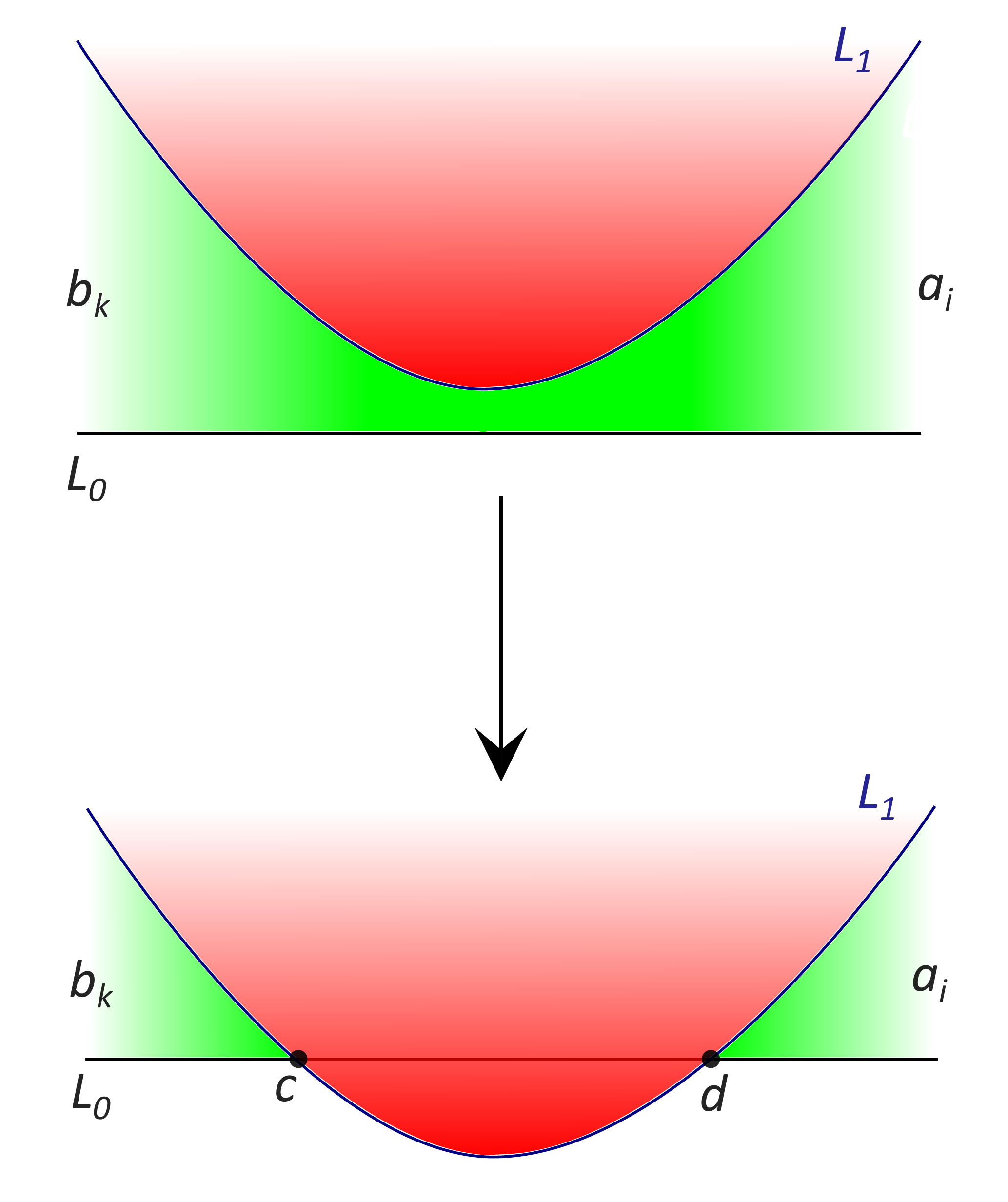}
\end{center}
\caption{The only old strips that break are the ones joining $a_i$'s to $b_k$'s}
\label{fig0}
\end{figure}

\subsection{The chain-isomorphism $H$}

Let us define the linear map 
$$H \colon D(t_0-\eps) \oplus \Z[T] \langle c,d \rangle \longrightarrow  D(t_0+\eps)$$
which satisfies:
\begin{itemize}
	\item for $1\leq k\leq m$, $H(b_k)=b_k$,
	\item $H(c)=c$ and $H(d)=\del_1 c$,
	\item finally, for $1 \leq i \leq n$, $H(a_i)=a_i + \langle \del_1 a_i,d \rangle c.$
\end{itemize}

Setting $\del_0 c \coloneqq d$ and $\del_0 d \coloneqq 0$ endows $D_0 \coloneqq D(t_0-\eps) \oplus \Z[T] \langle c,d \rangle$ with a structure of filtered chain-complex. Denote $D_1 \coloneqq D(t_0+\eps)$.
The function $H$ is the subject of the following proposition:

\begin{proposition}\label{chain map H}
$H$ is a chain-isomorphism from $(D_0,\del_0)$ to $(D_1,\del_1)$.
\end{proposition}

Lemma \ref{lemma3} ensures that the matrix representing $H$ in the basis $(b_m,\dots,b_1,$ $d, c, a_1,\dots,a_n)$ is upper-triangular with diagonal coefficients 1, hence invertible.

In addition, we need to verify that the following equation holds:
\begin{equation}\label{chainiso}
	\del_1 \circ H = H \circ \del_0.
\end{equation}

At first glance, \eqref{chainiso} is true on the generators $b_k, \ 1 \leq k \leq m$ and for $c$ and $d$.
Proving it for $a_i, \ 1 \leq i \leq n,$ requires some work. 

Our first step is the following:

\begin{proposition}\label{prop4}
Let $i\in\{1,\dots,n\}$. Then
\begin{equation*}
    \langle \del_1 a_i,c \rangle = \sum_{j=1}^{i-1} \langle \del_1 a_i,a_j\rangle \langle \del_1 a_j,d\rangle
\end{equation*}
\end{proposition}

\begin{proof}

First, using Proposition \ref{prop3} we see that 
\begin{equation}\label{eq3}
    \del_1 a_i = \del_0 a_i + \langle \del_1 a_i,c \rangle c + \langle \del_1 a_i,d \rangle d + \sum_{k=1}^m \langle \del_1 a_i,b_k \rangle b_k.
\end{equation}
We then apply $\del_1$ to Equation \eqref{eq3} and use $\del_1^2=0$:
\begin{equation}\label{eq4}
    0 = \del_1 \del_0 a_i + \langle \del_1 a_i,c \rangle \del_1 c + \langle \del_1 a_i,d \rangle \del_1 d + \sum_{k=1}^m \langle \del_1 a_i,b_k \rangle \del_1 b_k.
\end{equation}

For $1 \leq j < i$ and $1 \leq k \leq m$, we define 
\begin{align*}
    & \lambda_j\coloneqq \langle \del_1 a_j,c \rangle, \\
    & \mu_j\coloneqq \langle \del_1 a_j,d \rangle, \\
    & \xi_{j,k}\coloneqq \langle \del_1 a_j, b_k \rangle, \\
    & \nu_j\coloneqq \langle  \del_0 a_i, a_j \rangle.
\end{align*}

Using $\del_0^2=0$ and Proposition \ref{prop3} again, the first term of Equation \eqref{eq4} may then be simplified:
\begin{align*}
    \del_1 \del_0 a_i 
    & = \sum_{j=1}^{i-1} \langle \del_0 a_i,a_j\rangle \del_1 a_j + \sum_{k=1}^{m} \langle \del_0 a_i,b_k\rangle \del_1 b_k \\
    & = \sum_{j=1}^{i-1} \nu_j \left( \del_0 a_j + \lambda_j c + \mu_j d + \sum_{k=1}^m \xi_{j,k} b_k \right) + \sum_{k=1}^{m} \langle \del_0 a_i,b_k\rangle \del_0 b_k \\
    & = \del_0 ^2 a_i + \sum_{j=1}^{i-1} \nu_j \left( \lambda_j c + \mu_j d + \sum_{k=1}^m \xi_{j,k} b_k \right) \\
    & = \sum_{j=1}^{i-1} \nu_j \left( \lambda_j c + \mu_j d + \sum_{k=1}^m \xi_{j,k} b_k \right).
\end{align*}

Equation \eqref{eq4} then becomes:
\begin{align*}
    0 = & \sum_{j=1}^{i-1} \nu_j \left( \lambda_j c + \mu_j d + \sum_{k=1}^m \xi_{j,k} b_k \right) \\
    & + \langle \del_1 a_i,c \rangle \del_1 c + \langle \del_1 a_i,d \rangle \del_1 d + \sum_{k=1}^m \langle \del_1 a_i,b_k \rangle \del_1 b_k.
\end{align*}

Using Lemma \ref{lemma3}, the projection on $\Z[T] \cdot d$ of this equation gives
\begin{equation*}
    0 = \sum_{j=1}^{i-1} \nu_j \mu_j + \langle \del_1 a_i,c \rangle. 
\end{equation*}

Replacing the $\nu_j$'s and the $\mu_j$'s, and using Equation \eqref{eq6} yields
\begin{equation*}
    \langle \del_1 a_i,c \rangle = \sum_{j=1}^{i-1} \langle \del_1 a_i,a_j\rangle \langle \del_1 a_j,d\rangle,
\end{equation*}
which proves the proposition.
\end{proof}

In particular, $\langle \del_1 a_1,c \rangle =0$.

\bigskip
This proposition tells us that the strips joining $a_i$ to $c$ do not matter, since the new coefficients get canceled out by the definition of $H$. Indeed,
\begin{align*}
	\langle H(\del_0 a_i),c \rangle
	& = \sum_{j<i} \langle \del_0 a_i,a_j \rangle \langle H(a_j),c\rangle \\
	& = \sum_{j<i} \langle \del_0 a_i,a_j \rangle \langle \del_1 a_j,d\rangle,
\end{align*}
so by Proposition \ref{prop4},
\begin{equation}\label{proj_c}
	\langle H(\del_0 a_i),c \rangle = \langle \del_1 a_i,c \rangle.
\end{equation}

The second step is the following.

\begin{proposition} \label{lastprop}

Let us denote $p\colon D(t_0+\eps) \rightarrow \Z[T]\langle b_m,\dots,b_1,d \rangle$ the canonical projection. Then, $p(\del_1 c) = \del_1 c$ and 
\begin{equation}
    p(\del_1 a_i) = p(\del_0 a_i) + \langle \del_1 a_i,d \rangle \del_1 c.
\end{equation}
\end{proposition}

\begin{proof}
At time $t' \coloneqq t_0+\eps$, for each strip $S_1$ joining $a_i$ to $d$, and for each strip $S_2$ joining $c$ to $b_k$, there was a strip $S_0$ at $t \coloneqq t_0-\eps$, which disappears at $t=t_0$, joining $a_i$ to $b_k$, which satisfies
\begin{equation*}
    (S_0,h)=(S_1,h)+(S_2,h).
\end{equation*}
Here, $(S,h)$ denotes the algebraic intersection number, or the number of preimages of $h$ by $S$, as discussed at the end of Subsection \ref{2.2}.


Let $\tilde{\mathcal{M}}_t(a_i,b_k)$ be the set of the strips that contribute to the difference $\langle \del_1 a_i,b_k\rangle - \langle \del_0 a_i,b_k \rangle$:
$$\langle \del_1 a_i,b_k\rangle - \langle \del_0 a_i,b_k \rangle = \sum_{S_0 \in \tilde{\mathcal{M}}_t(a_i,b_k)} T^{(S_0,h)}.$$

The function $(S_1,S_2) \mapsto S_0$ described above gives rise to a bijection
$$ \mathcal{M}_{t'}(a_i,d) \times \mathcal{M}_{t'}(c,b_k) \overset{\sim}{\longrightarrow} \tilde{\mathcal{M}}_t(a_i,b_k).$$
Therefore, we see that for all $k\in\{1,\dots,m\}$, we have
\begin{align*}
\langle \del_1 a_i,b_k\rangle - \langle \del_0 a_i,b_k \rangle
&= \sum_{S_0 \in \tilde{\mathcal{M}}_t(a_i,b_k)} T^{(S_0,h)}  \\
&= \sum_{ (S_1,S_2) \in \mathcal{M}_{t'}(a_i,d) \times \mathcal{M}_{t'}(c,b_k) } T^{(S_1,h)+(S_2,h)} \\
&= \left( \sum_{S_1\in \mathcal{M}_{t'}(a_i,d)} T^{(S_1,h)} \right) \left( \sum_{S_2\in \mathcal{M}_{t'}(c,b_k)} T^{(S_2,h)} \right) \\
&= \langle \del_1 a_i,d \rangle \langle \del_1 c,b_k \rangle.
\end{align*}

By Lemma \ref{lemma3}, the coefficient in front of $d$ of $\del_0 a_i + \langle \del_1 a_i, d \rangle \del_1 c$ is exactly $\langle \del_1 a_i, d \rangle$, so combining it with the result of the equation above proves the proposition. This proof is summed up in Fig. \ref{fig2}.
\end{proof}

\begin{figure}[ht]  
\begin{center}
\includegraphics[scale = .25] {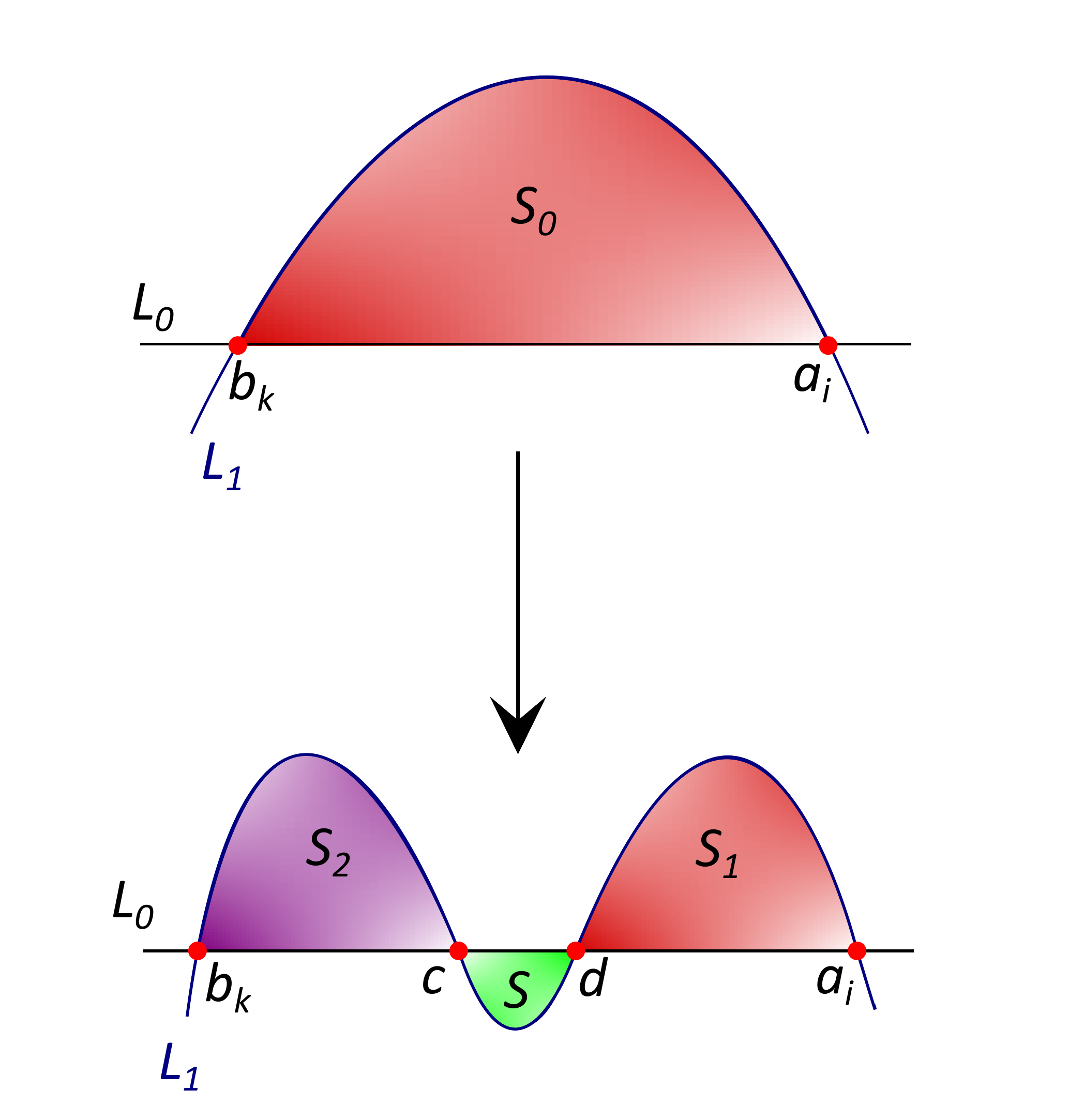}
\end{center}
\caption{The strip joining $a_i$ to $b_k$ breaks down into three parts}
\label{fig2}  
\end{figure}

Let us then prove Proposition \ref{chain map H}:

\begin{proof}

As explained in the discussion after the statement of the proposition, we only have to show that 
\begin{equation}\label{handle-slide}
    \del_1 H(a_i)=H(\del_0 a_i).
\end{equation}
This will be a consequence of the definition of $H$ and the following equation:
\begin{equation}\label{last}
	\del_1 a_i = H(\del_0 a_i) + \langle \del_1 a_i,d \rangle \del_1 c.
\end{equation}

Since $p(H(\del_0 a_i)))=p(\del_0 a_i)$, we already have by Proposition \ref{lastprop}
\begin{equation}\label{1st}
	p(\del_1 a_i) = p \big( H (\del_0 a_i) + \langle \del_1 a_i,d \rangle \del_1 c \big).
\end{equation}

We then only need to check if the projections on $c$ and on the $a_j$'s of both terms of \eqref{last} hold.

The projection on $a_j$ is given by Proposition \ref{prop3}:
\begin{equation}\label{2nd}
	\langle \del_1 a_i,a_j \rangle = \langle \del_0 a_i, a_j \rangle = \langle H(\del_0 a_i), a_j \rangle + \langle \del_1 a_i,d \rangle \langle \del_1 c, a_j \rangle,
\end{equation}
because $\langle \del_1 c, a_j \rangle = 0$.

Finally, Equation \eqref{proj_c} in the discussion after Proposition \ref{prop4} shows that 
\begin{equation}\label{3rd}		
	\langle \del_1 a_i,c \rangle = \langle H(\del_0 a_i), c \rangle + \langle \del_1 a_i,d \rangle \langle \del_1 c, c \rangle,
\end{equation}
where $\langle \del_1 c, c \rangle = 0$.

By the three equations \eqref{1st}, \eqref{2nd} and \eqref{3rd}, we see that \eqref{last} holds. 
Therefore, we have proved Proposition \ref{chain map H}.
\end{proof}

\subsection{Proof of Theorem \ref{thm}}\label{sub}

We are now able to prove Theorem \ref{thm}:

\begin{proof}
First of all, it is clear that $D(t)$ defines a piecewise continuous complex away from the birth and death moments. By genericity, we may assume that only a finite number of birth/death moments occur, and that only two generators appear/disappear simultaneously. 

Let $t_1<\dots<t_n$ be the birth/death moments. Denote $t_0\coloneqq 0$ and $t_{n+1}\coloneqq 1$. We will use the time interval
$$\Delta t \coloneqq \inf_{0 \leq i \leq n} t_{i+1} - t_i.$$

Let $1 \leq i \leq n$; we may assume that there is a birth at $t=t_i$. We are going to define a piecewise continuous family $\tilde{D}(t)$ such that $\tilde{D}(t) = D(t)$ when $t<t_i - \frac{\Delta t}{3}$ or $t>t_i + \frac{\Delta t}{3}.$
 
Denote $c(t)$ and $d(t)$ the two new-born generators of $D(t)$ at time $t>t_i$, with $\ell_t (c(t)) < \ell_t(d(t))$.

Let us define the family of filtered complexes
$$S \coloneqq \Z[T] \langle c,d \rangle,$$
endowed with the differential $\del$ that satisfies
$\del c = d$ and $\del d =0$, and with constant action $l$ such that $(c,d)$ is a compatible basis and 
$$l(c)=l(d) \coloneqq \lim_{s \to t_i^+} \ell_s (c) = \lim_{s \to t_i^+} \ell_s (d) .$$

Now, let us define the interval
$$I_i \coloneqq \left( t_i - \frac{\Delta t}{4} , t_i \right],$$

and then the family of complexes
$$(\tilde{D}(t),\tilde{\del}_h(t), \tilde{\ell}_t) \coloneqq 
\begin{cases}
	(D(t),\del_h(t),\ell_t) & \text{ if } t \not\in I_i, \\
	(D(t_i-\frac{\Delta t}{4}),\del_h(t_i-\frac{\Delta t}{4}),\ell_t) \oplus (S,\del,l)& \text{ if } t \in I_i, \\
\end{cases}$$
for $t \in [t_i-\frac{\Delta t}{3},t_i+\frac{\Delta t}{3}]$.

Unlike $D(t)$, the family $\tilde{D}(t)$ defines a true piecewise continuous family in the sense of Definition \ref{pwc}. Indeed:
\begin{itemize}
	\item $D(t_i)$ is not a filtered chain complex, but $D(t_i-\frac{\Delta t}{4})$ is, and it is chain-isomorphic to every $D(t)$ with $t\in \mathrm{int} \ I_i$, so the definition makes sense for all $t$, even at the bifurcation times $t_i$ and $t_i-\frac{\Delta t}{4}$;
	\item for $t<t_i - \frac{\Delta t}{4}$, $t>t_i$ or $t_i - \frac{\Delta t}{4} < t < t_i$, the complexes are chain-isomorphic with continuous action, as seen before;
	\item at time $t=t_i - \frac{\Delta t}{4}$, the complex $D(t)$ obviously undergoes a birth;
	\item at time $t=t_i$, the complex $D(t)$ undergoes a handle-slide, via the map $H$ which was proven to be suitable in Proposition \eqref{chain map H}.
\end{itemize} 

Therefore, the theorem is now proved for the family $\tilde{D}(t)$.
\end{proof}

\subsection{Proof of Theorem \ref{thm1}} \label{sec}

In order to prove Theorem \ref{thm1}, we have to make the link between the complexes $C_\beta(t)$ and $D(t)$. This link is described by the following lemma.

\begin{lemma}\label{ev}
Let $\beta \in\{0,1\}$. The evaluation map 
$$\mathrm{ev}_\beta \colon D(t) \longrightarrow C_\beta(t)$$
sending $T$ to $\beta$ is a chain-map that is action-preserving.

Moreover, if $D(t)$ has a Barannikov basis $B$, then $\mathrm{ev}_\beta(B)$ is a Barannikov basis of $C_\beta(t)$. In this case, we have in addition
$$\mathcal{B}(D(t))=\mathcal{B}(C_\beta(t)).$$
\end{lemma}

\begin{proof}
By definition of $\del_h(t)$, for each $x \in L \cap L_t$ we have
$$\mathrm{ev}_\beta (\del_h(t) x) = \del_\beta(t) x = \del_\beta(t) \mathrm{ev}_\beta x.$$
Since $\del_h(t)$ is $\Z[T]$-linear and $\mathrm{ev}_\beta \colon \Z[T] \rightarrow \Z$ is a ring homomorphism, this relation even holds for every $x\in D(t)$, so $\mathrm{ev}_\beta$ is a chain-map.
It obviously sends the canonical basis of $D(t)$ on the canonical basis of $C_\beta(t)$, so it preserves action.
It thus preserves compatible bases.

Let $B=(a_1(T),\dots,a_N(T),\del_h a_1(T),\dots,\del_h a_N(T),c_1(T),\dots,c_M(T))$ be a Barannikov basis of $D(t)$. Being a compatible basis, $\mathrm{ev}_\beta(B)$ is a compatible basis of $C_\beta(t)$. Moreover, we have
\begin{align*}
	& \forall i\in \{1,\dots,N\}, \quad \del_\beta \mathrm{ev}_\beta ( a_i(T)) = \mathrm{ev}_\beta ( \del_h a_i(T) ), \\
	& \forall j\in \{1,\dots,M\}, \quad \del_\beta \mathrm{ev}_\beta ( c_j(T)) = \mathrm{ev}_\beta ( \del_h c_j(T) ) = 0.
\end{align*}
Therefore, $\mathrm{ev}_\beta(B)$ is a Barannikov basis of $C_\beta(t)$.
\end{proof}


Finally, we use Lemma \ref{ev} and Theorem \ref{thm} to prove Theorem \ref{thm1}, which is a direct consequence of the following result.:

\begin{proposition}\label{cor}
Suppose $D(0)$ is a standard complex. Then, besides in some small neighborhood of its bifurcation times, $D(t)$ is a family of standard complexes, and its barcode satisfies 
$$\mathcal{B}(D(t))=\mathcal{B}(C_1(t))=\mathcal{B}(C_0(t)).$$
\end{proposition}

\begin{proof}
Let $\tilde{D}(t)$ be a modified version of $D(t)$ given by Theorem \ref{thm}.
Let $0<t_1<\dots<t_n<1$ be the bifurcation times. 

Between two bifurcations, the complexes are canonically chain-isomorphic, so if $\tilde{D}(t)$ is standard for some $t\in (t_i,t_{i+1})$, then all the $\tilde{D}(t')$ for all other $t' \in (t_i,t_{i+1})$ are standard as well.

We use induction over $i$ to prove that all the complexes $\tilde{D}(t)$ are standard.
Suppose that for all $t \in [0,t_i), \ \tilde{D}(t)$ is a standard complex.\\
Let $B \coloneqq (a_1,\dots,a_N,\del_0 a_1,\dots \del_0 a_N, c_1,\dots,c_M)$ be the Barannikov basis of $\tilde{D}(t_i-\eps)$.

Suppose there is a birth at time $t_i$. Then, there is a filtered decomposition
$$\tilde{D}(t_i+\eps) = \tilde{D}(t_i-\eps) \oplus S,$$
and therefore, after a proper reordering $B \sqcup (c,d)$ is the Barannikov basis of $\tilde{D}(t_i+\eps)$.

If there is a death at time $t_i$, by reversing the course of time we end up in the above case.

Suppose there is a handle-slide at time $t_i$. Denote 
$$H \colon (\tilde{D}(t_i-\eps),\del_0) \longrightarrow (\tilde{D}(t_i+\eps), \del_1)$$
the chain-isomorphism that corresponds to the handle-slide. Thus, $H(B)$ is a Barannikov basis of $(\tilde{D}(t_i+\eps),\del_1)$. Indeed,
\begin{align*}
& \forall i \in \{1,\dots,N\}, \quad \del_1 H(a_i) = H( \del_0 a_i ), \\
& \forall j \in \{1,\dots,M\}, \quad \del_1 H(c_j) = H(\del_0 c_j) = 0,
\end{align*}
and this proves that for all $t \in [0,t_{i+1}), \ \tilde{D}(t)$ is a standard complex. Moreover, it coincides with $D(t)$ out of a small neighborhood of the bifurcation times.

Finally, we use Lemma \ref{ev} to see that the barcode of $D(t)$ is $\mathcal{B}(C_1(t))$.
\end{proof}

\section*{Acknowledgment}
\addcontentsline{toc}{section}{Acknowledgment}

This article has been written during and quickly after an internship in Uppsala, where I have been working for 5 months under the supervision of Georgios Dimitroglou Rizell. Mr Dimitroglou had me introduced to symplectic geometry and Floer theory, and provided me with with a conjecture that I managed to prove and present in this paper. During the whole process, he helped me get an intuitive understanding of the topic as well as write rigorous definitions, statements and proofs. 

Therefore, I want to express my deep gratitude to Mr Dimitroglou; it has been a real pleasure to work together until the end of this project.

\printbibliography

\end{document}